\documentclass{article}

\usepackage{
	amsmath, 
	amsthm, 
	amssymb,
	fancyhdr, 
	setspace,
	tikz-cd,
}
\usepackage{upgreek}

\newcommand{\za}{\alpha}
\newcommand{\zb}{\beta}

\newcommand{\zd}{\delta}
\newcommand{\zg}{\gamma}

\usepackage{lscape}
\usepackage{color}
\usepackage[color,matrix,arrow]{xy}
\usepackage{setspace}
\xyoption{all}

\DeclareMathOperator{\Ext}{Ext}
\DeclareMathOperator{\Hom}{Hom}
\DeclareMathOperator{\pd}{pd}
\DeclareMathOperator{\id}{id}
\DeclareMathOperator{\add}{add}
\DeclareMathOperator{\Gen}{Gen}

\usepackage{avant}
\usepackage{amsmath, amssymb, amsfonts, color, tikz, bbm, txfonts, euscript}
\usepackage{amssymb}
\usepackage{enumerate}
\usepackage{bm}
\newtheorem{theorem}{Theorem}[section]
\newtheorem{lemma}[theorem]{Lemma}

\newtheorem{prop}[theorem]{Proposition}
\newtheorem{cor}[theorem]{Corollary}
\theoremstyle{definition}
\newtheorem{mydef}[theorem]{Definition}
\newtheorem{example}[theorem]{Example}

\begin{document}

\thispagestyle{empty}

\title{Auslander Algebras Which Are Tilted}
\author{Stephen Zito\thanks{2010 Mathematics Subject Classification. Primary 16G10; Secondary 16G70.\ Keywords: Auslander algebras; tilted algebras.}}
        
\maketitle

\begin{abstract}
Let $A$ be an Auslander algebra of global dimension equal to $2$.  We provide a necessary and sufficient condition for $A$ to be a tilted algebra.  In particular, $A$ is a tilted algebra if and only if $\pd_{A}(\tau_{A}\Omega_{A}DA)\leq1$.
\end{abstract}

\section{Introduction}
We set the notation for the remainder of this paper. All algebras are assumed to be finite dimensional over an algebraically closed field $k$.  If $\Lambda$ is a $k$-algebra then denote by $\mathop{\text{mod}}\Lambda$ the category of finitely generated right $\Lambda$-modules and by $\mathop{\text{ind}}\Lambda$ a set of representatives of each isomorphism class of indecomposable right $\Lambda$-modules.  Given $M\in\mathop{\text{mod}}\Lambda$, the projective dimension of $M$ in $\mathop{\text{mod}}\Lambda$ is denoted by $\pd_{\Lambda}M$ and its injective dimension by $\id_{\Lambda}M$.  We denote by $\add M$ the smallest additive full subcategory of $\mathop{\text{mod}}\Lambda$ containing $M$, that is, the full subcategory of $\mathop{\text{mod}}\Lambda$ whose objects are the direct sums of direct summands of the module $M$.  We let $\tau_{\Lambda}$ and $\tau^{-1}_{\Lambda}$ be the Auslander-Reiten translations in $\mathop{\text{mod}}\Lambda$.  $D$ will denote the standard duality functor $\Hom_k(-,k)$.  Given a $\Lambda$-module $M$, denote the kernel of a projective cover by $\Omega_{\Lambda} M$.  $\Omega_{\Lambda} M$ is called the syzygy of $M$.  Denote the cokernel of an injective envelope by $\Omega_{\Lambda}^{-1}M$.  $\Omega_{\Lambda}^{-1}M$ is called the cosyzygy of $M$.  Finally,  let $\mathop{\text{gl.dim}}\Lambda$ stand for the global dimension and $\mathop{\text{domdim}}\Lambda$ stand for the dominant dimension of an algebra $\Lambda$ (see Definition $\ref{def3}$).

Let $\Lambda$ be an algebra of finite type and $M_1,M_2,\cdots,M_n$ be a complete set of representatives of the isomorphism classes of indecomposable $\Lambda$-modules.  Then  $A=\text{End}_{\Lambda}(\oplus_{i=1}^nM_i)$ is the Auslander algebra of $\Lambda$.  Auslander in $\cite{AU}$ characterized the algebras which arise this way as algebras of global dimension at most $2$ and dominant dimension at least $2$.
\par
 Let $\mathcal{C}_{\Lambda}$ be the full subcategory of $\mathop{\text{mod}}\Lambda$ consisting of all modules generated and cogenerated by the direct sum of representatives of the isomorphism classes of all indecomposable projective-injective $\Lambda$-modules (see Definition $\ref{Q}$).  When $\mathop{\text{gl.dim}}\Lambda=2$, Crawley-Boevey and Sauter showed in $\cite{CBS}$ that the algebra $\Lambda$ is an Auslander algebra if and only if there exists a tilting $\Lambda$-module $T_{\mathcal{C}}$ in $\mathcal{C}_{\Lambda}$. 
\par
Recent work by Nguyen, Reiten, Todorov, and Zhu in $\cite{NRTZ}$ showed the existence of such a tilting module is equivalent to the dominant dimension being at least $2$ without any condition on the global dimension of $\Lambda$.  They also gave a precise description of such a tilting module.
\par
Tilting theory is one of the main themes in the study of the representation theory of algebras.  Given a tilting module over a hereditary algebra, its endomorphism algebra is called a tilted algebra.  For brevity, we refer the reader to  $\cite{ASS}$ for relevant definitions.  In this paper, we give a necessary and sufficient condition for an Auslander algebra $A$ of $\mathop{\text{gl.dim}}A=2$ to be tilted.  We note that an Auslander algebra $A$ of $\mathop{\text{gl.dim}}A\leq1$ is hereditary and thus tilted. 

\begin{theorem}$\emph{(Theorem~2.2)}.$
Suppose $A$ is an Auslander algebra of $\mathop{\emph{gl.dim}}A=2$.  Then $A$ is a tilted algebra if and only if $\pd_{A}(\tau_{A}\Omega_{A}DA)\leq1$.
\end{theorem}
The motivation for this result was the consideration of the following problem: Given an algebra $\Lambda$ of finite type and $A$ its Auslander algebra, what properties of $\mathop{\text{mod}}\Lambda$ can be deduced from properties of $\mathop{\text{mod}}A$?  If more structure is placed on $A$, such as $A$ being tilted, what can we say about $\Lambda$?  We think it would be an interesting and fruitful line of research. 

\subsection{Properties of the subcategory $\mathcal{C}_{\Lambda}$}
Let $\Lambda$ be an algebra.
\begin{mydef}
\label{Gen/Cogen}
Let $M$ be a $\Lambda$-module.  We define $\mathop{Gen} M$ to be the class of all modules $X$ in $\mathop{\text{mod}}\Lambda$ generated by $M$, that is, the modules $X$ such that there exists an integer $d\geq0$ and an epimorphism $M^d\rightarrow X$ of $\Lambda$-modules.  Here, $M^d$ is the direct sum of $d$ copies of $M$.  Dually, we define $\mathop{Cogen}M$ to be the class of all modules $Y$ in $\mathop{\text{mod}}\Lambda$ cogenerated  by $M$, that is, the modules $Y$ such that there exist an integer $d\geq0$ and a monomorphism $Y\rightarrow M^d$ of $\Lambda$-modules.
\end{mydef}
\begin{mydef}
\label{Q}
Let $\tilde{Q}$ be the direct sum of representatives of the isomorphism classes of all indecomposable projective-injective $\Lambda$-modules.  Let $\mathcal{C}_{\Lambda}:=(\text{Gen}\tilde{Q})\cap(\text{Cogen}\tilde{Q})$ be the full subcategory consisting of all modules generated and cogenerated by $\tilde{Q}$.
\end{mydef}
When $\mathop{\text{gl.dim}}\Lambda=2$, Crawley-Boevey and Sauter showed the following characterization of Auslander algebras.
\begin{lemma}$\emph{\cite[Lemma~1.1]{CBS}}$
\label{CBS}
If $\mathop{\emph{gl.dim}}\Lambda=2$, then $\mathcal{C}_{\Lambda}$ contains a tilting-cotilting module if and only if $\Lambda$ is an Auslander algebra.
\end{lemma}
Nguyen, Reiten, Todorov, and Zhu in $\cite{NRTZ}$ showed the existence of such tilting and cotilting modules without any condition on the global dimension of $\Lambda$ and gave a precise description of such tilting and cotilting modules in $\mathcal{C}_{\Lambda}$.  We first recall the definition of the dominant dimension of an algebra.
\begin{mydef}
\label{def3}
Let $\Lambda$ be an algebra and let
\begin{center}
$0\rightarrow\Lambda_{\Lambda}\rightarrow I_0\rightarrow I_1\rightarrow\ I_2\rightarrow\cdots$
\end{center}
be a minimal injective resolution of $\Lambda$.  Then $\mathop{\text{domdim}}\Lambda=n$ if $I_i$ is projective for $0\leq i\leq n-1$ and $I_n$ is not projective.  If all $I_n$ are projective, we say $\text{domdim}\Lambda=\infty$.
\end{mydef}

\begin{theorem}$\emph{\cite[Theorem~1]{NRTZ}}$
\label{DomDim}

Let $\mathcal{C}_{\Lambda}$ be the full subcategory consisting of all modules generated and cogenerated by $\tilde{Q}$
\begin{enumerate}
\item[\emph{(1)}] The following statements are equivalent:
\begin{enumerate}
\item[\emph{(a)}] $\mathop{\emph{domdim}}\Lambda\geq2$
\item[\emph{(b)}] $\mathcal{C}_{\Lambda}$ contains a tilting module $T_{\mathcal{C}}$
\item[\emph{(c)}] $\mathcal{C}_{\Lambda}$ contains a cotilting module $C_{\mathcal{C}}$.
\end{enumerate}
\item[\emph{(2)}] If a tilting module $T_{\mathcal{C}}$ exists in $\mathcal{C}_{\Lambda}$, then $T_{\mathcal{C}}\cong\tilde{Q}\oplus(\bigoplus_i\Omega_{\Lambda}^{-1}P_i)$, where $\Omega_{\Lambda}^{-1}P_i$ is the cosyzygy of $P_i$ and the direct sum is taken over representatives of the isomorphism classes of all indecomposable projective non-injective $\Lambda$-modules $P_i$.
\item[\emph{(3)}] If a cotilting module $C_{\mathcal{C}}$ exists in $\mathcal{C}_{\Lambda}$, then $C_{\mathcal{C}}\cong\tilde{Q}\oplus(\bigoplus_i\Omega_{\Lambda}I_i)$ where $\Omega_{\Lambda}I_i$ is the syzygy of $I_i$ and the direct sum is taken over representatives of the isomorphism classes of all indecomposable injective non-projective $\Lambda$-modules $I_i$.
\end{enumerate}

\end{theorem}

\subsection{Tilted Algebras}
The endomorphism algebra of a tilting module over a hereditary algebra is said to be $\emph{tilted}$.
There exist many characterizations of tilted algebras.  One such sufficient characterization of tilted algebras was established by Zito in $\cite{ZITO2}$.  It involves the notion of a splitting torsion pair.
Given a tilting $\Lambda$-module $T$, it is well known that $T$ induces a torsion pair $(\mathcal{T},\mathcal{F})$  
 where $(\text{Gen}T,\mathcal{F}(T))=(\mathcal{T}(T),\text{Cogen}(\tau_{\Lambda} T))$. In particular, $D\Lambda\in\mathcal{T}(T)$.  We refer the reader to $\cite{ASS}$ for more details on torsion pairs.
 \par
 We say a torsion pair $(\mathcal{T},\mathcal{F})$
 is $\it{splitting}$ if every indecomposable $\Lambda$-module belongs to either $\mathcal{T}$ or $\mathcal{F}$.  We have the following characterization of torsion pairs which are splitting.  
  \begin{prop}$\emph{\cite[VI,~Proposition~1.7]{ASS}}$
 \label{split}
 Let $(\mathcal{T},\mathcal{F})$ be a torsion pair in $\mathop{\emph{mod}}\Lambda$.  The following are equivalent:
 \begin{enumerate}
 \item[$\emph{(a)}$] $(\mathcal{T},\mathcal{F})$ is splitting.
 \item[$\emph{(b)}$] If $M\in\mathcal{T}$, then $\tau_{\Lambda}^{-1}M\in\mathcal{T}$.
 \item[$\emph{(c)}$] If $N\in\mathcal{F}$, then $\tau_{\Lambda}N\in\mathcal{F}$.
 \end{enumerate}
 \end{prop}

\begin{prop}$\emph{\cite[Proposition~2.1]{ZITO2}}$
\label{prop1}
If there exists a tilting module $T$ in $\mathop{\emph{mod}}\Lambda$ such that the induced torsion pair $(\mathcal{T}(T),\mathcal{F}(T))$ is splitting and $\id_{\Lambda}X\leq1$ for every $X\in\mathcal{T}(T)$, then $\Lambda$ is tilted.
\end{prop}

\section{Main Result}
We begin with a preliminary result.  When $\mathop{\text{domdim}}\Lambda=2$, Theorem $\ref{DomDim}$ guarantees the existence of a tilting module $T_{\mathcal{C}}$ in $\mathcal{C}_{\Lambda}$.  The next proposition provides a necessary and sufficient condition for the induced torsion pair $(\mathcal{T}(T_{\mathcal{C}}),\mathcal{F}(T_{\mathcal{C}}))$ to be splitting.
\begin{prop}
\label{prop2}
The torsion pair $(\mathcal{T}(T_{\mathcal{C}}),\mathcal{F}(T_{\mathcal{C}}))$ is splitting if and only if $\pd_{\Lambda}X\leq1$ for every $X\in\mathcal{F}(T_{\mathcal{C}})$.
\end{prop}
\begin{proof}
Assume $(\mathcal{T}(T_{\mathcal{C}}),\mathcal{F}(T_{\mathcal{C}}))$ is splitting and let $X\in\mathcal{F}(T_{\mathcal{C}})$.  Then $\tau_{\Lambda}X\in\mathcal{F}(T_{\mathcal{C}})$ by Proposition $\ref{split}$.
Since $\mathcal{F}(T_{\mathcal{C}})=\text{Cogen}(\tau_{\Lambda}T_{\mathcal{C}})$, we have a monomorphism $f:\tau_{\Lambda}X\rightarrow(\tau_{\Lambda}T_{\mathcal{C}})^d$ for some integer $d\geq0$.  If $\text{Hom}_{\Lambda}(D\Lambda,\tau_{\Lambda}X)\neq0$, then, since $D\Lambda\in\mathop{\text{Gen}}(T_{\mathcal{C}})$ and $f$ is a monomorphism, we would have a non-zero map $h:(T_{\mathcal{C}})^{d'}\rightarrow\tau_{\Lambda}T_{\mathcal{C}}$ for some integer $d'\geq0$.  This leads to a contradiction since the Auslander-Reiten formulas guarantee $\Ext_{\Lambda}^1(T_{\mathcal{C}},T_{\mathcal{C}})\cong D\Hom_{\Lambda}(T_{\mathcal{C}},\tau_{\Lambda}T_{\mathcal{C}})=0$.  We conclude $\text{Hom}_{\Lambda}(D\Lambda,\tau_{\Lambda}X)=0$ and $\pd_{\Lambda}X\leq1$.
\par
Assume $\pd_{\Lambda}X\leq1$ for every $X\in\mathcal{F}(T_{\mathcal{C}})$.  Then $\text{Hom}_{\Lambda}(D\Lambda,\tau_{\Lambda}X)=0$.  We know $T_{\mathcal{C}}\cong\tilde{Q}\oplus(\bigoplus_i\Omega_{\Lambda}^{-1}P_i)$ where $\Omega_{\Lambda}^{-1}P_i$ is the cosyzygy of $P_i$ and the direct sum is taken over representatives of the isomorphism classes of all indecomposable projective non-injective $\Lambda$-modules $P_i$ by Theorem $\ref{DomDim}$ (2).  Since $\tilde{Q}$ is injective and $(\bigoplus_i\Omega_{\Lambda}^{-1}P_i)\in\Gen(\tilde{Q})$ by the definition of $\mathcal{C}_{\Lambda}$, we must have $\Hom_{\Lambda}(T_{\mathcal{C}},\tau_{\Lambda}X)=0$.  Since $\tau_{\Lambda}X\in\mathcal{F}(T_{\mathcal{C}})$, we conclude $(\mathcal{T}(T_{\mathcal{C}}),\mathcal{F}(T_{\mathcal{C}}))$ is splitting by Proposition $\ref{split}$ (c).

\end{proof}
We are now ready for our main result.  We use freely homological results pertaining to algebras $\Lambda$ of $\mathop{\text{gl.dim}}\Lambda=2$.  See $\cite{ZITO}$ for further reference.
\begin{theorem}
\label{Main}
Let $A$ be an Auslander algebra of $\mathop{\emph{gl.dim}}A=2$.  Then $A$ is tilted if and only if $\pd_{A}(\tau_{A}\Omega_{A}DA)\leq1$.
\end{theorem}
\begin{proof}
If $A$ is tilted, then $\pd_{A}(\tau_{A}\Omega_{A}DA)\leq1$ since $\id_{A}(\tau_{A}\Omega_{A}DA)=2$ for every $X\in\add(\tau_{A}\Omega_{A}DA)$.  Next, assume $\pd_{A}(\tau_{A}\Omega_{A}DA)\leq1$.  Since $A$ is an Auslander algebra, $\mathcal{C}_{A}$ contains a tilting-cotilting module $T_{\mathcal{C}}$ by Lemma $\ref{CBS}$.  Let $(\mathcal{T}(T_{\mathcal{C}}),\mathcal{F}(T_{\mathcal{C}}))$ be the induced torsion pair and consider an indecomposable module $X\in\mathcal{T}(T_{\mathcal{C}})$.  Since $\pd_{A}T_{\mathcal{C}}\leq1$, the Auslander-Reiten formulas imply $\Hom_{A}(X,\tau_{A}T_{\mathcal{C}})=0$. Since $T_{\mathcal{C}}$ is cotilting, Theorem $\ref{DomDim}$ (3) says $T_{\mathcal{C}}\cong\tilde{Q}\oplus(\bigoplus_i\Omega_{A}I_i)$ where all $I_i$ are non-projective.  We have $\tau_{A}T_{\mathcal{C}}=\tau_{A}(\bigoplus_i\Omega_{A}I_i)$ since $\tau_{A}\tilde{Q}=0$.
Thus, $\Hom_{A}(X,\tau_{A}(\bigoplus_i\Omega_{A}I_i))=0$ implies $\id_{A}X\leq1$.  Since $X$ was arbitrary, we conclude $\id_{A}X\leq1$ for every $X\in\mathcal{T}(T_{\mathcal{C}})$.
\par
Next, consider an indecomposable module $Y\in\mathcal{F}(T_{\mathcal{C}})$.  We have the following equalities $\mathcal{F}(T_{\mathcal{C}})=\text{Cogen}(\tau_{A}T_{\mathcal{C}})=\text{Cogen}(\tau_{A}(\bigoplus_i\Omega_{A}I_i))$ where the last equality follows from Theorem $\ref{DomDim}$ (3).  If $\pd_{A}Y=2$, then $\Hom_{A}(\tau_{A}^{-1}\Omega_{A}^{-1}A,Y)\neq0$.  Since $Y\in\text{Cogen}(\tau_{A}(\bigoplus_i\Omega_{A}I_i))$, we must have $\Hom_{A}(\tau_{A}^{-1}\Omega_{A}^{-1}A,\tau_{A}(\bigoplus_i\Omega_{A}I_i))\neq0$.  This implies $\pd_{A}(\tau_{A}(\bigoplus_i\Omega_{A}I_i))=2$.  Since this contradicts our original assumption, we must have $\pd_{A}Y\leq1$.  Since $Y$ was arbitrary, we conclude $\pd_{A}Y\leq1$ for every $Y\in\mathcal {F}(T_{\mathcal{C}})$.
\par
Since $\pd_{A}Y\leq1$ for every $Y\in\mathcal{F}(T_{\mathcal{C}})$, we know $(\mathcal{T}(T_{\mathcal{C}}),\mathcal{F}(T_{\mathcal{C}}))$ is splitting by Proposition $\ref{prop2}$.  We have also shown $\id_{A}X\leq1$ for every $X\in\mathcal{T}(T_{\mathcal{C}})$.  Thus, Proposition $\ref{prop1}$ implies $A$ is a tilted algebra.

\end{proof}
Following $\cite{HRS}$, we say an algebra $\Lambda$ is $quasi$-$tilted$ if $\mathop{\text{gl.dim}}\Lambda\leq2$ and, for every $X\in\mathop{\text{ind}}\Lambda$, we have $\pd_{\Lambda}X\leq1$ or $\id_{\Lambda}X\leq1$. The following corollary shows if $\Lambda$ is an Auslander and quasi-tilted algebra, then $\Lambda$ must be tilted. 
\begin{cor}
Let $A$ be an Auslander algebra of $\mathop{\emph{gl.dim}}=2$.  Then $A$ is quasi-tilted if and only if $A$ is tilted.
\end{cor}
\begin{proof}
Assume $A$ is quasi-tilted.  Since $\id_{A}X=2$ for every $X\in\add(\tau_{A}(\bigoplus_i\Omega_{A}I_i))$, we must have $\pd_{A}X\leq1$ for every $X\in\add(\tau_{A}(\bigoplus_i\Omega_{A}I_i))$.  We conclude from Theorem $\ref{Main}$ that $A$ is tilted.  If $A$ is tilted, then $A$ is trivially quasi-tilted.
\end{proof}

\section{Examples}
In this section, we illustrate our main result with two small examples.
\begin{example}  
Consider the Auslander algebra $A$ given by the following quiver with relations
$$\begin{array}{cc}
\xymatrix{1\ar[r]^\za&2\ar[r]^\zb&3}
&\quad\za\zb=0.\end{array}$$
We have $\tau_A\Omega_ADA=S(3)$, where $S(3)$ denotes the simple module at vertex $3$.  Notice, $\pd_AS(3)=0$ and $A$ is a tilted algebra.

\end{example}
\begin{example}
Consider the Auslander algebra $A$ given by the following quiver with relations 
$$\begin{array}{cc}
\xymatrix{1\ar[r]^\za&2\ar[r]^\zb&3\ar[r]^\zg&4\ar[r]^\zd&5}
&\quad\za\zb=\zg\zd=0.\end{array}$$
We have $\tau_{A}\Omega_{A}DA=P(5)\oplus S(3)$, where $P(5)$ denotes the simple projective at vertex $5$ and $S(3)$ denotes the simple module at vertex $3$.  Notice, $\pd_{A}S(3)=2$ and $A$ is not tilted.

\end{example}

\noindent Mathematics Faculty, University of Connecticut-Waterbury, Waterbury, CT 06702, USA
\it{E-mail address}: \bf{stephen.zito@uconn.edu}

\end{document}